\documentclass[10pt]{article}
\usepackage{amsmath, amssymb, amsfonts, amsthm, epsfig, float, graphicx, sectsty, stix}
\usepackage[all]{xy}
\title{Quantifying lawlessness in finitely generated groups}
\author{Henry Bradford}

\newtheorem{thm}{Theorem}[section]
\newtheorem{lem}[thm]{Lemma}
\newtheorem{propn}[thm]{Proposition}
\newtheorem{coroll}[thm]{Corollary}
\newtheorem{defn}[thm]{Definition}
\newtheorem{ex}[thm]{Example}
\newtheorem{notn}[thm]{Notation}

\newtheorem{rmrk}[thm]{Remark}
\newtheorem{qu}[thm]{Question}
\newtheorem{prob}[thm]{Problem}

\DeclareMathOperator{\Aut}{Aut}

\DeclareMathOperator{\Homeo}{Homeo}

\DeclareMathOperator{\im}{im}
\DeclareMathOperator{\rk}{rk}
\DeclareMathOperator{\rst}{top}

\DeclareMathOperator{\Stab}{Stab}

\DeclareMathOperator{\supp}{supp}
\DeclareMathOperator{\Sym}{Sym}
\DeclareMathOperator{\Wr}{Wr}

\begin{document}

\maketitle

\begin{abstract}
We introduce a quantitative 
notion of lawlessness for finitely generated groups, 
encoded by the \emph{lawlessness growth function} 
$\mathcal{A}_{\Gamma} : \mathbb{N} \rightarrow \mathbb{N}$. 
We show that $\mathcal{A}_{\Gamma}$ is bounded iff 
$\Gamma$ has a nonabelian free subgroup. 
By contrast we construct, 
for any nondecreasing unbounded function $f: \mathbb{N} \rightarrow \mathbb{N}$, 
an elementary amenable lawless groups for which 
$\mathcal{A}_{\Gamma}$ grows more slowly that $f$. 
We produce torsion lawless groups for which $\mathcal{A}_{\Gamma}$ 
is at least linear using Golod-Shafarevich theory, 
and give some upper bounds on $\mathcal{A}_{\Gamma}$ 
for Grigorchuk's group and Thompson's group $\mathbf{F}$. 
We note some connections between $\mathcal{A}_{\Gamma}$ 
and quantitative versions of residual finiteness. 
Finally, we also describe a function $\mathcal{M}_{\Gamma}$ quantifying the property 
of $\Gamma$ having no mixed identities, 
and give bounds for nonabelian free groups. 
By contrast with $\mathcal{A}_{\Gamma}$, 
there are \emph{no} groups for which $\mathcal{M}_{\Gamma}$ is bounded: 
we prove a universal lower bound on $\mathcal{M}_{\Gamma}(n)$ 
of the order of $\log (n)$. 
\end{abstract}

\section{Introduction}

A \emph{law} for a group $\Gamma$ is a non-trivial 
word-map which vanishes identically on $\Gamma$. 
$\Gamma$ is \emph{lawless} if it has no laws. 
The goal of this article is to introduce and study a 
quantitative notion of lawlessness. 

\subsection{Statement of results}

Throughout, $\Gamma$ is a group generated by a finite set $S$. 
Given a nontrivial reduced word $w$ which is not a law for $\Gamma$, 
the \emph{complexity} of $w$ in $\Gamma$ is the minimal 
length of a tuple in $\Gamma$ not evaluating to the identity under $w$ 
(where the length of a tuple of elements of $\Gamma$ 
is the sum of the lengths of those elements 
in the word-metric induced by $S$ on $\Gamma$). 
The \emph{lawlessness growth function} 
$\mathcal{A}_{\Gamma} ^S : \mathbb{N} \rightarrow \mathbb{N}$ 
of $\Gamma$ then sends $n \in \mathbb{N}$ to the maximal complexity 
occurring among the words of length at most $n$. 
It is easy to see that if $\Gamma$ has a nonabelian free subgroup, 
then $\Gamma$ is lawless and $\mathcal{A}_{\Gamma} ^S$ 
is a bounded function. 
Our first observation is that this is the only way 
that bounded lawlessness growth can arise. 

\begin{thm} \label{bddLGthm}
Suppose $\Gamma$ is lawless. 
Then $\Gamma$ has a nonabelian free subgroup iff $\mathcal{A}_{\Gamma} ^S$ 
is bounded. 
\end{thm}

Theorem \ref{bddLGthm} immediately begs the question of 
how slowly $\mathcal{A}_{\Gamma} ^S$ can grow for 
a lawless finitely generated group which contains no 
nonabelian free subgroup. 
We give a satisfying answer to this question, 
constructing examples where the growth can be as slow as desired. 

\begin{thm} \label{EEffthm}
Let $f : \mathbb{N} \rightarrow \mathbb{N}$ be an unbounded nondecreasing function 
with $f(1) \geq 2$. 
There exists an elementary amenable lawless group $\Gamma$, 
generated by a finite set $S$ such that for all $n\in\mathbb{N}$, 
$\mathcal{A}_{\Gamma} ^S (n) \leq f (n)$. 
\end{thm}

The groups $\Gamma$ can all be built out of a fixed lawless 
elementary amenable group, using a wreath product construction. 
Being amenable, $\Gamma$ contains no nonabelian free subgroup, 
so that by Theorem \ref{bddLGthm} $\mathcal{A}_{\Gamma} ^S$ is unbounded. 
Next we turn to groups of fast lawlessness growth. 

\begin{thm} \label{GSmainthm}
There exists a lawless finitely generated group $\Gamma$ 
with $\mathcal{A}_{\Gamma} ^S (n) \gg n$. 
\end{thm}

The group $\Gamma$ is constructed using Golod-Shafarevich theory: 
it has long been known that G-S theory enables the construction of infinite 
(indeed lawless) finitely generated $p$-torsion groups. 
We run this argument in an effective manner: 
ensuring that the orders of elements grow as slowly as possible 
as a function of word-length. 

We then turn to the problem of estimating $\mathcal{A}_{\Gamma} ^S$ 
for specific well-known examples of lawless groups without 
nonabelian free subgroups. 
The first example we investigate is 
Grigorchuk's (first) group of automorphisms of the binary rooted tree. 
This is a torsion group (so has no nonabelian free subgroup) 
and has every finite $2$-group as a subgroup (so has no laws). 

\begin{thm} \label{Grigmainthm}
Let $\Gamma$ be the first Grigorchuk group. 
Then there exists $C>0$ such that: 
\begin{equation}
n^{2/3} \ll \mathcal{A}_{\Gamma} ^S (n) \ll \exp (Cn)\text{.}
\end{equation}
\end{thm}

The upper bound in Theorem \ref{Grigmainthm} 
is based on embedding iterated wreath products 
of $C_2$ into $\Gamma$, close to the identity, 
and bounds on the lengths of the shortest laws for these groups. 
The lower bound derives 
from a bound on the rate of growth of element-orders in $\Gamma$, 
due to Bartholdi and \u{S}uni\'{k}. 
Being as it is based on complexity-bounds for 
words of a very specific type (power words), 
we expect that our lower bound in Theorem \ref{Grigmainthm} 
is far from best-possible; 
indeed, we conjecture that Grigorchuk's group has 
super-polynomial lawlessness growth. 

Our second example is Thompson's group $\mathbf{F}$ 
of homeomorphisms of the interval. 
Brin and Squier \cite{BrinSqui} prove that 
$\mathbf{F}$ has no nonabelian free subgroup and satisfies no law. 
Making their arguments effective, we obtain the following conclusion. 

\begin{thm} \label{ThompMainThm}
Thompson's group $\mathbf{F}$ satisfies 
$\mathcal{A}_{\mathbf{F}} ^S (n) \ll n$. 
\end{thm}



If the word $w$ does not vanish under evaluation at the tuple $\mathbf{g}$, 
then whenever $\pi : \Gamma \rightarrow Q$ is a homomorphism 
in whose kernel $w (\mathbf{g})$ does not lie, 
$w$ is not a law for $Q$. 
For residually finite lawless groups $\Gamma$, 
this observation draws a link between the behaviour of 
$\mathcal{A}_{\Gamma} ^S$, the lengths of the shortest laws holding in finite groups, 
and the \emph{residual finiteness growth function} 
$\mathcal{F}_{\Gamma} ^S : \mathbb{N} \rightarrow \mathbb{N}$ 
due to Bou-Rabee \cite{BouRab}. 
For instance, in \cite{BradThom} the author and A. Thom proved 
the following result. 

\begin{thm} \label{BradThomRFThm}
Suppose $\Gamma$ contains a nonabelian free subgroup and let $\delta > 0$. 
There exists $C>0$ such that: 
\begin{equation*}
\mathcal{F}_{\Gamma} ^S (n) \geq C n^{3/2} / \log (n)^{9/2 + \delta}\text{.}
\end{equation*}
\end{thm}

By the preceding paragraph, 
we deduce the following conclusion valid for all lawless groups, 
of which Theorem \ref{BradThomRFThm} is a special case 
(since by Theorem \ref{bddLGthm} groups satisfying the hypothesis 
of Theorem \ref{BradThomRFThm} are precisely those for which 
$\mathcal{A}_{\Gamma} ^S$ is bounded). 

\begin{thm} \label{RFLawlessThmIntro}
Let $\Gamma$ be a finitely generated lawless group 
and let $\delta > 0$. There exists $C>0$ such that: 
\begin{equation*}
\mathcal{F}_{\Gamma} ^S \big( n \mathcal{A}_{\Gamma} ^S (n) \big) 
\geq C n^{3/2} / \log (n)^{9/2 + \delta}\text{.}
\end{equation*}
\end{thm}

We also have an analogous bound for the residual $p$-finiteness growth. 
Finally we note a connection between $\mathcal{A}_{\Gamma} ^S$ 
and decision problems for recursively presented groups. 

\begin{thm} \label{WordProbThm}
If $\Gamma$ is a finitely generated recursively presented lawless group 
with decidable word problem, 
then $\mathcal{A}_{\Gamma}$ is computable. 
\end{thm}

If, instead of word-maps, we consider word-maps with coefficients, 
then we have the class of MIF groups (instead of lawless groups) 
and we can define an \emph{MIF growth function} $\mathcal{M}_{\Gamma}$. 
By contrast with Theorem \ref{bddLGthm}, we have a nonconstant universal lower bound. 

\begin{thm} \label{MIFLBThm}
There are no groups of bounded MIF growth. 
Indeed, for any finitely generated group $\Gamma$, 
$\mathcal{M}_{\Gamma} (l) \gg \log (l)$. 
\end{thm}

We also have an effective version of the fact (due to G. Baumslag) that 
nonabelian free groups are MIF. 

\begin{thm} \label{MIFfreeThm}
Let $\Gamma$ be a nonabelian free group of finite rank. 
Then $\mathcal{M}_{\Gamma} (n) \ll n \log(n)$. 
\end{thm}

\subsection{Background and structure of the paper}

There are some important similarities between the class of groups 
satisfying laws and the class of amenable groups: 
both classes are closed under taking subgroups; quotients and extensions; 
both contain all finite and abelian groups, 
and neither contains $F_2$.  
An important difference between the two classes is that 
an ascending union of amenable groups is amenable, 
whereas an ascending union of groups with laws may be lawless. 
Our Theorem \ref{EEffthm} shows how stark a distinction is opened up by this property: 
by taking (an extension of) an ascending union of finite groups, 
we may construct amenable groups which are nonetheless ``very'' lawless. 

In the other direction, there are known examples of nonamenable groups satisfying laws: 
for instance Adian showed that free Burnside groups of sufficiently large odd 
exponent are nonamenable. 
Such examples are, however, quite exotic. 
It is, for example, an open question whether every 
nonamenable \emph{residually finite} group is lawless \cite{CornMann}. 
Our Theorem \ref{RFLawlessThmIntro} shows that there is a tension 
between residual finiteness and lawlessness, in a quantitative sense: 
a group cannot be simultaneously ``very'' residually finite 
and ``very'' lawless. 

There has also been great interest in recent 
years in estimating the lengths of the shortest laws for \emph{finite} groups. 
One may view the following problem as the ``finitary'' version 
of the problem of estimating $\mathcal{A}_{\Gamma}$ for a fixed lawless group: 
given a sequence $(G_n)$ of finite groups which do not share a common law, 
can we estimate the asymptotic behaviour of the length of the shortest law 
$w_n \in F_2$ for $G_n$ as $n \rightarrow \infty$? 
This latter problem has been explored with $G_n$ taken to be 
the symmetric group $\Sym(n)$ \cite{KozTho}; 
a finite simple group of Lie type of bounded rank \cite{BraThoLie}, 
or the direct product of all groups of order at most $n$ 
(see Theorem \ref{BThomThm} below). 
We exploit analogies between the finite and infinite settings 
several times in this paper, most notably in Theorems \ref{EEffthm}; 
\ref{Grigmainthm} and \ref{RFLawlessThmIntro}. 

The paper is structured as follows. 
In Section \ref{PrelSection} we define $\mathcal{A}_{\Gamma} ^S$ 
and establish its basic properties. 
We also prove Theorem \ref{WordProbThm} in this Section. 
Sections \ref{BddSection}-\ref{ThompSect} are devoted, respectively, 
to the proofs of Theorems \ref{bddLGthm}-\ref{ThompMainThm}. 
In Section \ref{RFSect} we establish connections between 
$\mathcal{A}_{\Gamma} ^S$ and quantitative versions of residual finiteness. 
In Section \ref{CoeffSect} we define $\mathcal{M}_{\Gamma} ^S$ 
and prove Theorems \ref{MIFLBThm} and \ref{MIFfreeThm}. 
We conclude with a selection of open problems, 
speculations and directions for future research. 

\section{Preliminaries} \label{PrelSection}

Let $\lvert \cdot \rvert_S : \Gamma \rightarrow \mathbb{N}$ 
be the word-length function induced on $\Gamma$ by $S$. 
Denote by $B_S (l)$ the set of elements of $\Gamma$ 
of length at most $l$. 
For $F_k$ the free group of rank $k$, 
we denote by $\lvert \cdot \rvert$ the word length 
on $F_k$ induced by a fixed free basis 
$X_k = \lbrace x_1,\ldots x_k\rbrace$, 
and may write $B(l)$ for $B_{X_k} (l)$. 
The word $w \in F_k$ is a \emph{law} for $\Gamma$ 
if it is non-trivial in $F_k$ and lies in the kernel 
of every homomorphism $F_k \rightarrow \Gamma$. 
The set of \emph{non-laws} for $\Gamma$ in $F_k$ is: 
\begin{equation*}
N_k (\Gamma) = F_k \setminus \bigcap_{\pi} \ker(\pi)
\end{equation*}
where the intersection is taken over all homomorphisms 
$\pi : F_k \rightarrow \Gamma$. 
For $w \in N_k (\Gamma)$, 
we define the \emph{complexity} of $w$ in $\Gamma$ 
(with respect to $S$) to be: 
\begin{equation*}
\chi_{\Gamma} ^S (w) 
= \min \big\lbrace \sum_{i=1} ^k \lvert g_i \rvert_S 
: \mathbf{g} \in \Gamma^k , w(\mathbf{g}) \neq 1 \big\rbrace\text{.}
\end{equation*}

For $W \subseteq N_k (\Gamma)$, 
the \emph{$W$-lawlessness growth function} (or anarchy growth; 
pandemonium growth etc.) of $\Gamma$ with respect to $S$ 
is defined to be: 
\begin{equation*}
\mathcal{A}_{\Gamma,W} ^S (l) 
= \max \big\lbrace \chi_{\Gamma} ^S (w) 
: w \in W; \lvert w \rvert \leq l \big\rbrace\text{.}
\end{equation*}
If $N_k (\Gamma) = F_k \setminus \lbrace 1 \rbrace$, 
we denote $\mathcal{A}_{\Gamma,F_k \setminus \lbrace 1 \rbrace} ^S$ 
by  $\mathcal{A}_{\Gamma} ^S$ and refer to it simply 
as the \emph{lawlessness growth of $\Gamma$} 
(with respect to $S$). 

\begin{lem} \label{supersetlawsineq}
If $W^{\prime} \subseteq W \subseteq N_k (\Gamma)$, then 
$\mathcal{A}_{\Gamma,W} ^S (l) \geq 
\mathcal{A}_{\Gamma,W^{\prime}} ^S (l)$
for all $l \in \mathbb{N}$. 
\end{lem}

\begin{notn}
For nondecreasing functions 
$F_1 , F_2 : \mathbb{N} \rightarrow \mathbb{N}$ 
we write $F_1 \preceq F_2$ if there exists $K \in \mathbb{N}$ 
such that $F_1 (l) \leq K F_2 (Kl)$ for all $l \in \mathbb{N}$, 
and $F_1 \approx F_2$ if $F_1 \preceq F_2$ and $F_2 \preceq F_1$. 
It is clear that $\approx$ is an equivalence relation. 
\end{notn}

\begin{rmrk} \label{freerkrmrk}
\normalfont
In our notation $\mathcal{A}_{\Gamma} ^S$ for the lawlessness 
growth function, we leave implicit the rank $k$ of the free group 
$F_k$ in which our words lie. 
We shall assume $k$ to be an arbitrary (but fixed) integer at least $2$. 
It turns out that very little is lost by making this assumption. 
First suppose that $\Gamma$ has no law in $F_k$ and $1 \leq k^{\prime} \leq k$. 
Embed $F_{k^{\prime}}$ into $F_k$ by extending 
the set of free variables (with $\iota$ the inclusion map). 
Then for $w \in F_{k^{\prime}}$ and $g_1 , \ldots, g_{k^{\prime}} \in \Gamma$, 
\begin{equation} \label{cxtyrkeq}
w(g_1 , \ldots, g_{k^{\prime}})=\iota(w)(g_1 , \ldots, g_{k^{\prime}},e,\ldots,e),\text{ so }
\chi_{\Gamma} ^S (w) = \chi_{\Gamma} ^S \big( \iota(w)\big)\text{.}
\end{equation}
We apply Lemma \ref{supersetlawsineq} with 
$W^{\prime} = \iota(F_{k^{\prime}}) \setminus \lbrace 1 \rbrace$ 
and $W = F_k \setminus \lbrace 1 \rbrace$ 
so that by (\ref{cxtyrkeq}),
\begin{center}
$\mathcal{A}_{\Gamma,F_k\setminus \lbrace 1 \rbrace} ^S (l) 
\geq \mathcal{A}_{\Gamma,\iota(F_{k^{\prime}})\setminus \lbrace 1 \rbrace} ^S (l) 
= \mathcal{A}_{\Gamma,F_{k^{\prime}}\setminus \lbrace 1 \rbrace} ^S (l)$
\end{center}
for all $l \in \mathbb{N}$. 
Conversely suppose that $2 \leq k^{\prime} \leq k$ 
and that $\Gamma$ has no laws in $F_{k^{\prime}}$. 
Let $\phi : F_k \hookrightarrow F_{k^{\prime}}$ 
be an embedding, and for $1 \leq i \leq k$ 
write $\phi (x_i) = v_i \in F_{k^{\prime}}$. 
For $w \in F_k$ non-trivial, 
$\phi(w)=w(v_1,\ldots,v_k)$ is not a law for $\Gamma$, 
so there are $\mathbf{g} \in \Gamma^{k^{\prime}}$ 
with $\phi(w)(\mathbf{g})
=w(v_1(\mathbf{g}),\ldots,v_k(\mathbf{g})) \neq e$. 
Hence the $k$-tuple $(v_1(\mathbf{g}),\ldots,v_k(\mathbf{g}))$ 
witnesses that $w$ is not a law for $\Gamma$, 
and that $\chi_{\Gamma} ^S (w) \leq C \chi_{\Gamma} ^S \big( \phi(w) \big)$, 
where $C = \max \lbrace \lvert v_i\rvert : 1 \leq i \leq k \rbrace$. 
Since $\lvert \phi(w) \rvert \leq C \lvert w \rvert$ 
and $\lvert v_i (\mathbf{g}) \rvert_S 
\leq C (\lvert g_1 \rvert_S + \cdots + \lvert g_k \rvert_S)$,
\begin{equation}
\mathcal{A}_{\Gamma,F_k\setminus \lbrace 1 \rbrace} ^S (l) 
\leq C k \mathcal{A}_{\Gamma,\phi(F_k)\setminus \lbrace 1 \rbrace} ^S (Cl) 
\leq C k \mathcal{A}_{\Gamma,F_{k^{\prime}}\setminus \lbrace 1 \rbrace} ^S (Cl)
\end{equation}
for all $l\in \mathbb{N}$. 
Thus $\mathcal{A}_{\Gamma,F_k\setminus \lbrace 1 \rbrace} ^S 
\approx \mathcal{A}_{\Gamma,F_{k^{\prime}}\setminus \lbrace 1 \rbrace} ^S$. 
\end{rmrk}

\begin{lem} \label{subgrpcomplem}
Let $\Delta \leq \Gamma$ be a subgroup generated by a finite set $T$. 
Then there exists $C>0$ such that for all $1 \neq w \in F_k$, 
if $w$ is not a law for $\Delta$, then: 
\begin{equation} \label{subgrpcxineq}
\chi_{\Delta} ^T (w) \geq C \chi_{\Gamma} ^S (w)\text{.}
\end{equation} 
Thus, if $W \subseteq N_k (\Delta)$, then 
$\mathcal{A}_{\Delta,W} ^T (l) \geq C \mathcal{A}_{\Gamma,W} ^S (l)$
for all $l \in \mathbb{N}$. 
\end{lem}

\begin{proof}
There exists $C > 0$ such that for all $g \in \Gamma$, 
$\lvert g \rvert_T \geq C \lvert g \rvert_S$. 
(\ref{subgrpcxineq}) follows immediately. 
\end{proof}

In view of Remark \ref{freerkrmrk} and Lemma \ref{subgrpcomplem}, 
we study $\mathcal{A}_{\Gamma} ^S$ up to $\approx$. 

\begin{coroll}
Let $S_1$ and $S_2$ be finite generating sets for $\Gamma$. 
Then $\mathcal{A}_{\Gamma} ^{S_1} \approx\mathcal{A}_{\Gamma} ^{S_2}$. 
\end{coroll}

\begin{lem}
Let $\pi : \Gamma \twoheadrightarrow Q$ be an epimorphism of groups. 
Then $N_k (Q) \subseteq N_k (\Gamma)$ and, for any $w \in N_k (Q)$, 
\begin{equation}
\chi_{\Gamma} ^S (w) \leq \chi_Q ^{\pi(S)} (w)\text{.}
\end{equation}
Thus, if $W \subseteq N_k (Q)$, 
then $\mathcal{A}_{\Gamma,W} ^S (l) 
\leq \mathcal{A}_{Q,W} ^{\pi(S)} (l)$ for all $l$. 
\end{lem}

\begin{ex} \label{TorsionEx}
Suppose that $\Gamma$ is a torsion group. 
The \emph{torsion growth function} of $\Gamma$ (with respect to $S$) 
is defined to be: 
\begin{equation*}
\pi_{\Gamma} ^S (n) = \max \lbrace o(g) : \lvert g \rvert_S \leq n \rbrace
\end{equation*}
and was introduced in \cite{GrigArt}. 
In case $\Gamma$ is a $p$-group of infinite exponent, 
there is an intimate connection between torsion growth 
and lawlessness growth: let $W = \lbrace x^{p^k}:k\in\mathbb{N}\rbrace \subset F_2$ be the set of $p$-power words.  
Then for $m,n \in \mathbb{N}$, 
\begin{equation*}
\mathcal{A}_{\Gamma,W} ^S (p^m) \geq n+1
 \Leftrightarrow \pi_{\Gamma} ^S(n) \leq p^m \text{.}
\end{equation*}
Known finitely generated torsion $p$-groups of infinite exponent 
include various \emph{branch groups}, 
for which the torsion growth has been estimated. 
We discuss this further in Section \ref{BranchSect}.  
\end{ex}

Our proof of Theorems \ref{EEffthm} and \ref{MIFLBThm} 
uses the following Proposition, 
which is Lemma 2.2 from \cite{KozTho}. 
For $w \in F_k$ and $\Gamma$ any group, let: 
\begin{center}
$Z(w,\Gamma) = \lbrace \mathbf{g} \in \Gamma^k : w(\mathbf{g})=e \rbrace$
\end{center}
be the \emph{vanishing set} of $w$ in $\Gamma$. 

\begin{propn} \label{KozThoProp}
Let $k \geq 2$ and let $w_1 , \ldots , w_m \in F_k$ be nontrivial. 
There exists $w \in F_k$ nontrivial such that, 
for any group $\Gamma$, 
\begin{center}
$Z(w,\Gamma) \supseteq Z(w_1,\Gamma) \cup \cdots \cup Z(w_m,\Gamma)$
\end{center}
and $\lvert w \rvert \leq 16 m^2 \max_i \lvert w_i \rvert$. 
\end{propn}

\begin{coroll} \label{KozThoCoroll}
Let $\Gamma$ be a lawless group and let $k \geq 2$. 
Then for all $l \geq 1$, there exist $g_1 , \ldots , g_k \in \Gamma$ 
such that, for all $v \in F_k$ nontrivial with $\lvert v \rvert \leq l$, 
$v (g_1 , \ldots , g_k) \neq e$.   
\end{coroll}

\begin{proof}
Suppose not. Let $w_1 , \ldots , w_m \in F_k$ be a list of all nontrivial 
reduced words of length at most $l$, 
and let $w$ be as in the conclusion of Proposition \ref{KozThoProp}. 
Then for every $g_1 , \ldots , g_k \in \Gamma$, 
there exists $1 \leq i \leq m$ such that 
$(g_1 , \ldots , g_k) \in Z(w_i,\Gamma) \subseteq Z(w,\Gamma)$, 
so $w (g_1 , \ldots , g_k)$, and therefore $w$ is a law for $\Gamma$. 
\end{proof}

\begin{proof}[Proof of Theorem \ref{WordProbThm}]
Let $S$ be a finite generating set for $\Gamma$. 
We are given an algorithm \texttt{WORDPROBLEM} 
which takes as input an element of $F(S)$ and determines 
whether or not it evaluates to the identity element of $\Gamma$. 
We describe an algorithm \texttt{COMPLEX} which takes as input 
a word $w \in F_k$ and returns $\chi_{\Gamma} ^S(w)$. 
Applying \texttt{COMPLEX} to all $w$ of length $\leq n$, 
we compute $\mathcal{A}_{\Gamma} ^S (n)$. 

At the $m$th step, 
\texttt{COMPLEX} either verifies that $\chi_{\Gamma} ^S(w) \leq m$ 
and terminates, 
or verifies that $\chi_{\Gamma} ^S(w) > m$ and proceeds to the 
$(m+1)$th step. We thus establish the exact value of $\chi_{\Gamma} ^S(w)$. 

The $m$th step of \texttt{COMPLEX} runs as follows. 
Let $B$ be a list of all elements of 
$F(S)$ of length at most $m$ (a finite set). 
For each ordered $k$-tuple $\mathbf{v}$ of (not necessarily distinct) 
elements of $B$, apply \texttt{WORDPROBLEM} to $w(\mathbf{v})$. 
Note that $\chi_{\Gamma} ^S(w) \leq m$ iff for some $\mathbf{v}$, 
$w(\mathbf{v})$ is non-trivial in $\Gamma$. 
\end{proof}

\section{Bounded lawlessness growth} \label{BddSection}

We recall two basic facts about free groups. 
Let $F$ be a free group of finite rank. 

\begin{lem} \label{freenrmllem}
Let $1 \neq N \vartriangleleft F$. 
\begin{itemize}
\item[(i)] If $\lvert F:N \rvert < \infty$, 
then $\rk(N) - 1 = \lvert F:N \rvert \cdot (\rk(F) - 1)$; 
\item[(ii)] If $\lvert F:N \rvert = \infty$ then $\rk(N) = \infty$. 
\end{itemize}
In particular, if $F$ is nonabelian then $N$ is not cyclic. 
\end{lem}

\begin{lem} \label{cycClem}
The centralizer of a non-trivial element of $F$ is a cyclic subgroup. 
\end{lem}

\begin{proof}[Proof of Theorem \ref{bddLGthm}]
Suppose $a,b \in \Gamma$ with $\langle a,b \rangle \cong F_2$. 
Let $C>0$ with $\lvert a \rvert_S + \lvert b \rvert_S \leq C$. 
Then for all $1 \neq w \in F_2$, $w(a,b) \neq 1$, 
so $\chi_{\Gamma} ^S (w) \leq C$. 
Thus $\mathcal{A}_{\Gamma} ^S$ is bounded. 

Conversely suppose (for a contradiction) that $\mathcal{G}_{\Gamma} ^S$ 
is bounded but that $\Gamma$ has no subgroup isomorphic to $F_2$. 
Let $r \in \mathbb{N}$ be minimal such that there exist 
$(g_1,h_1),\ldots,(g_r,h_r) \in \Gamma \times \Gamma$ 
such that for all nontrivial $w(x,y) \in F_2$, 
there exists $1 \leq i \leq r$ such that $w(g_i,h_i) \neq 1$. 
Note that $r \geq 2$ (else $\langle g_1,h_1 \rangle \cong F_2$). 
By minimality of $r$, there exists $1 \neq w_r \in F_2$ 
such that: 
\begin{center}
$w_r (g_1,h_1),\ldots,w_r(g_{r-1},h_{r-1})=1$. 
\end{center}
Let $N = \lbrace w\in F_2 : w(g_1,h_1)=1 \rbrace$. 
Then $N \vartriangleleft F_2$ 
(it is the kernel of the homomorphism sending an ordered basis 
of $F_2$ to $g_1,h_1$) and $1 \neq N$ 
(for instance, $w_r \in N$). 
But $N \leq C_{F_2} (w_r)$, 
as otherwise any nontrivial word of the form $[w_r,v]$, 
for $v \in N$, would be a law for $\Gamma$ (contradicting hypothesis). 
By Lemma \ref{cycClem} $N$ is cyclic, 
contradicting Lemma \ref{freenrmllem}
\end{proof}

\section{Amenable groups of slow lawlessness growth}

In this Section we prove Theorem \ref{EEffthm}. 
We start with an elementary combinatorial fact. 

\begin{lem} \label{SparseFnctLem}
There exist increasing functions 
$p,q : \mathbb{N} \rightarrow \mathbb{N} \cup \lbrace 0 \rbrace$ 
such that $p(1)=q(1)=0$ and, for all $i,j,k,l \in \mathbb{N}$, 
\begin{itemize}
\item[(i)] For $r = p$ or $q$, 
if $r(j)-r(k)=r(l)-r(i)$ then either (a) $i=k$ and $j=l$, 
or (b) $j=k$ and $i=l$; 

\item[(ii)] If $q(j)-p(k)=q(l)-p(i)$ then $i=k$ and $j=l$. 

\end{itemize}
\end{lem}

\begin{proof}
We construct $p$ and $q$ via a recursive process. 
We have $p(1)=q(1)=0$ and set $p(2)=1$, $q(2)=2$, 
so that (i) and (ii) hold for $i,j,k,l \in \lbrace 1,2\rbrace$. 
Supposing we have constructed 
$p(1)\leq q(1) < p(2)\leq q(2) < \ldots < p(n)\leq q(n)$ such that 
(i) and (ii) hold for all $1 \leq i,j,k,l \leq n$, 
we define $p(n+1)$ such that: 
\begin{equation} \label{SparseIneq1}
p(n+1) \geq p(n) + q(n) + 1
\end{equation}
and then define $q(n+1)$ by: 
\begin{equation} \label{SparseIneq2}
q(n+1) \geq p(n+1) + q(n) + 1
\end{equation}
so that $p(n)<q(n)<p(n+1)<q(n+1)$. 
We check that (i) and (ii) hold for all $1 \leq i,j,k,l \leq n+1$. 

For (i), consider the equation $p(j)-p(k)=p(l)-p(i)$. 
We may assume that $l \geq j \geq k$. 
Since $p$ is increasing, if $j=k$ then $i=l$ and if $l=j$ then $i=k$, 
so we may assume $l > j > k$. We claim that $p(j)-p(k) < p(l)-p(i)$. 
Subject to our assumptions, the maximal value of 
$p(j)-p(k)$ is attained for $j=l-1$ and $k=1$, 
while the minimal value of $p(l)-p(i)$ is attained for $i=l-1$, 
and $p(l-1)-p(1) < p(l)-p(l-1)$ by construction. 
The argument for $q$ is exactly the same. 

For (ii), consider the equation $q(j)-p(k)=q(l)-p(i)$. We may assume $l > j$ 
(as before, WLOG $l \geq j$ and if $l=j$ then $i=k$ as $p$ is increasing). 
We distinguish two cases. In the case that $l \geq i$, 
we have $j \geq k$ (since the quantity in the equation is non-negative) 
so that: 
\begin{equation*}
q(j)-p(k) \leq q (l-1) < q(l)-p(l) \leq q(l)-p(i)\text{.}
\end{equation*}
In the second case, $l < i$ so $j < k < i$ and: 
\begin{equation*}
q(l)-q(j) \leq q(l) < p(i)-p(i-1) \leq p(i)-p(k)
\end{equation*}
so that $q(l)-p(i) < q(j)-p(k)$. 
\end{proof}

\begin{rmrk}
Though we shall not need it in the sequel, it is not difficult to see 
that a minimal solution to the inequalities 
(\ref{SparseIneq1}) and (\ref{SparseIneq2}) 
yields functions $p$ and $q$ growing at most exponentially. 
\end{rmrk}

Fix a lawless elementary amenable group $\Delta$ 
(not necessarily finitely generated). 
For example, we may take $\Delta$ to be the direct sum of any sequence 
of finite groups which do not have a common law, such as 
the sequence of all finite symmetric groups \cite{KozTho}. 
By Corollary \ref{KozThoCoroll}, for each $l \geq 1$ 
we have $g_l,h_l \in \Delta$ such that for all $v \in F_2$ nontrivial, 
if $\lvert v \rvert \leq l$ then $v (g_l , h_l) \neq e$. 
Let $L : \mathbb{N} \rightarrow \mathbb{N}$ be an increasing 
function to be determined. 
Let $\hat{g} , \hat{h} : \mathbb{Z} \rightarrow \Delta$ be defined as follows: 
\begin{itemize}
\item[(i)] For each $n \in \mathbb{N}$, 
$\hat{g} (p(n)) = g_{L(n)}$ and $\hat{h} (q(n)) = h_{L(n)}$; 
\item[(ii)] $\hat{g}(m)=e$ for $m \notin \im (p)$; 
\item[(iii)] $\hat{h}(m)=e$ for $m \notin \im (q)$, 
\end{itemize}
where $p$ and $q$ are as in Lemma \ref{SparseFnctLem}. 
Let $G = \Delta \Wr \mathbb{Z} = \Delta^{\mathbb{Z}} \rtimes \mathbb{Z}$ 
be the \emph{unrestricted} wreath product of $\Delta$ and $\mathbb{Z}$, 
with the $\mathbb{Z}$-factor being generated by $t$. 
Let $S=S(L)=\lbrace \hat{g} , \hat{h} , t \rbrace$ 
and define $\Gamma = \Gamma(L) = \langle S(L) \rangle \leq G$. 
We claim that, for an appropriate choice of the function $L$, 
the group $\Gamma$ satisfies the conclusion of Theorem \ref{EEffthm}, 
the proof of which we divide between the next two results. 

\begin{thm}
For any $L$, $\Gamma$ is elementary amenable. 
\end{thm}

\begin{proof}
Let $N = \Gamma \cap \Delta^{\mathbb{Z}}$. 
Then $N \vartriangleleft \Gamma$ with $\Gamma/N \cong \mathbb{Z}$, 
so it suffices to check that $N$ is elementary amenable. 
$N/[N,N]$ is a countable abelian group, so it suffices to check that 
$[N,N]$ is elementary amenable. 

The group $N$ is generated by 
$X = \lbrace \hat{g}^{t^n} , \hat{h}^{t^n} : n\in\mathbb{Z} \rbrace$, 
so $[N,N]$ is normally generated in $N$ by 
$Y=\lbrace [\hat{f}_1 , \hat{f}_2] : \hat{f}_1 , \hat{f}_2 \in X \rbrace$. 
If $Y \subseteq \bigoplus_{\mathbb{Z}} \Delta$, then 
$[N,N] \leq \bigoplus_{\mathbb{Z}} \Delta$ 
(since $\bigoplus_{\mathbb{Z}} \Delta \vartriangleleft \Delta^{\mathbb{Z}}$), 
so that $[N,N]$ is elementary amenable 
(being a subgroup of the elementary amenable group $\bigoplus_{\mathbb{Z}} \Delta$). 

We therefore claim that $Y \subseteq \bigoplus_{\mathbb{Z}} \Delta$. 
Recall that the \emph{support} of $\hat{f} \in \Delta^{\mathbb{Z}}$ 
is $\supp (\hat{f}) = \lbrace n\in\mathbb{Z} : \hat{f}(n)\neq e \rbrace$, 
so that $\bigoplus_{\mathbb{Z}} \Delta$ is precisely the group 
of finite-support elements of $\Delta^{\mathbb{Z}}$. 
We assert that: 
\begin{itemize}
\item[(a)] for every $\hat{f}_1 , \hat{f}_2 \in \Delta^{\mathbb{Z}}$, 
$\supp([\hat{f}_1 , \hat{f}_2]) \subseteq \supp (\hat{f}_1) \cap \supp (\hat{f}_2)$; 
\item[(b)] For all pairs of distinct elements $\hat{f}_1 , \hat{f}_2 \in X$, 
$\lvert \supp (\hat{f}_1) \cap \supp (\hat{f}_2) \rvert \leq 1$, 
\end{itemize}
whence the desired claim. 
Observation (a) is clear, and (b) follows from Lemma \ref{SparseFnctLem}: 
$\supp (\hat{g}) \subseteq \im (p)$ and $\supp (\hat{h}) \subseteq \im (q)$, 
so a point in $\supp (\hat{g}^{t^m}) \cap \supp (\hat{h}^{t^n})$ 
is $p(i) + m = q(l) +n$ for some $i,l \in \mathbb{N}$. 
A second point in the intersection would yield a second pair $j,k \in \mathbb{N}$ 
satisfying $p(k) + m = q(j) +n$, so that $q(j)-p(k)=q(l)-p(i)$, 
contradicting Lemma \ref{SparseFnctLem} (ii). 
We argue similarly for $\supp (\hat{g}^{t^m}) \cap \supp (\hat{g}^{t^n})$ 
and $\supp (\hat{h}^{t^m}) \cap \supp (\hat{h}^{t^n})$ 
for $m \neq n$, using Lemma \ref{SparseFnctLem} (i). 
\end{proof}

\begin{thm}
For every unbounded nondecreasing function $f : \mathbb{N} \rightarrow \mathbb{N}$ 
with $f(1) \geq 2$, 
there exists $L$ such that for all $n \in \mathbb{N}$, 
$\mathcal{A}_{\Gamma(L)} ^{S(L)}(n) \leq f(n)$. 
\end{thm}

\begin{proof}
Let $m\in \mathbb{N}$ and let $w \in F_2$ be nontrivial 
with $\lvert w \rvert \leq L(m)$. 
Then: 
\begin{equation*}
 w\big( \hat{g}^{t^{q(m)-p(m)}},\hat{h} \big)\big( q(m) \big) 
 = w( g_{L(m)},h_{L(m)} ) \neq e
\end{equation*} 
so that $w$ is not a law for $\Gamma (L)$ and:
\begin{equation} \label{WreathCxtyEqn}
\chi_{\Gamma(L)} ^{S(L)} (w) \leq \big\lvert \hat{g}^{t^{q(m)-p(m)}} \big\rvert_{S(L)} + \big\lvert \hat{h} \big\rvert_{S(L)}
\leq 2 \big( q(m)-p(m) + 1 \big)\text{.}
\end{equation} 
We require $L(m) \in \mathbb{N}$ 
sufficiently large that $f(L(m)) \geq 2 \big( q(m+1)-p(m+1) + 1 \big)$ 
(possible since $f$ is unbounded). 

Now let $n \in \mathbb{N}$. 
Suppose first that $n \geq L(1)$. 
Let $m \in \mathbb{N}$ with $L(m)\leq n \leq L(m+1)$. Then: 
\begin{align*}
\mathcal{A}_{\Gamma(L)} ^{S(L)} (n) \leq \mathcal{A}_{\Gamma(L)} ^{S(L)} (n) 
& \leq \mathcal{A}_{\Gamma(L)} ^{S(L)} \big(L(m+1)\big) \\
& \leq 2 \big( q(m+1)-p(m+1) + 1 \big) \text{ (by (\ref{WreathCxtyEqn}))}\\
& \leq f\big(L(m)\big) \\
& \leq f(n). 
\end{align*}
On the other hand, if $n \leq L(1)$ then for any $w \in F_2$ nontrivial 
with $\lvert w \rvert \leq n$, $w (\hat{g},\hat{h}) \neq e$, 
so $\chi_{\Gamma(L)} ^{S(L)} (w) \leq 2$ 
and $\mathcal{A}_{\Gamma(L)} ^{S(L)} (n) \leq 2$. 
\end{proof}

\section{Golod-Shafarevich groups of linear lawlessness growth}

The goal of this Section is to prove Theorem \ref{GSmainthm}. 
Throughout this Section $p$ is an arbitrary (but fixed) prime number. 
We follow the treatment of Golod-Shafaverich groups from Chapter 3 of \cite{Ersh}. 
For $\Gamma$ an abstract group, denote by $\hat{\Gamma}_{(p)}$ 
the \emph{pro-$p$ completion} of $\Gamma$. 
Let $\mathbb{F}_p \lAngle U_k \rAngle$ be the 
algebra of power-series in the non-commuting variables 
$U_k = \lbrace u_1 , \ldots , u_k \rbrace$ over $\mathbb{F}_p$. 
For $f \in \mathbb{F}_p \lAngle U_k \rAngle$, 
the \emph{degree} $\deg(f)$ of $f$ is the minimal length 
of a monomial occuring in $f$ with nonzero coefficient. 
Let $\hat{F}_k$ be a free pro-$p$ group on the finite set 
$X_k = \lbrace x_1 , \ldots , x_k \rbrace$. 
There is a continuous monomorphism (the \emph{Magnus embedding}) 
$\mu :\hat{F}_k \hookrightarrow\mathbb{F}_p \lAngle U_k \rAngle^{\ast}$ 
extending $x_i \mapsto 1 + u_i$ \cite{Ersh}. 

\begin{defn} \label{GSDefn}
\begin{itemize}
\item[(i)] A pro-$p$ group $G$ is \emph{Golod-Shafarevich} 
if it admits a pro-$p$ presentation $G = \langle X_k \mid R \rangle$ 
such that there exists $\tau \in (0,1)$ with: 
\begin{equation} \label{GSineq}
1- k \tau + \sum_{r \in R} \tau^{D(r)} < 0,
\end{equation}
where for $r \in \hat{F}$, 
$D(r) = \deg(\mu(r)-1)$ is the \emph{degree} of $r$; 
\item[(ii)] An abstract group $\Gamma$ is 
Golod-Shafarevich if $\hat{\Gamma}_{(p)}$ is a Golod-Shafarevich 
pro-$p$ group. 
\end{itemize}
\end{defn}

All the relations in the presentations we construct will be $p$-powers, 
and estimating the degrees of these is easily done, 
direct from the definition of the function $D$. 

\begin{lem} \label{degpowerlem}
For all $w \in \hat{F}_k$, $D(w^p) = p D(w)$. 
\end{lem}

The next, basic Lemma immediately implies that if $G = \langle X_k \mid R \rangle$ 
is an \emph{abstract} group presentation satisfying (\ref{GSineq}), 
then $G$ is an abstract Golod-Shafarevich group. 

\begin{lem}
Let $\Gamma$ be an abstract group, 
and suppose $\langle S \mid R \rangle$ 
is an abstract group presentation for $\Gamma$. 
Then $\langle S \mid R \rangle$ is a pro-$p$ 
presentation for $\hat{\Gamma}_{(p)}$. 
\end{lem}

Finally, we need a guarantee that the groups we construct are indeed lawless 
(see \cite{Zel} p.224). 

\begin{thm} \label{Zelmanovthm}
Let $G$ be a Golod-Shafarevich pro-$p$ group. 
Then $G$ has a non-abelian free subgroup. 
\end{thm}

\begin{coroll} \label{GSlawlesscoroll}
Suppose $\Gamma$ is an (abstract) Golod-Shafarevich group. 
Then $\Gamma$ is lawless. 
\end{coroll}

\begin{proof}
Let $1 \neq w \in F_2$. 
Let $\mathcal{N}$ be the inverse system 
of normal subgroups of $p$-power index in $\Gamma$, 
and let $G = \varprojlim_{N \in \mathcal{N}} \Gamma/N$ 
be the pro-$p$ completion of $\Gamma$. 
By Theorem \ref{Zelmanovthm} there exist $g,h \in G$ 
freely generating a rank-$2$ free subgroup of $G$. 
Then there exists $N \in \mathcal{N}$ 
such that the image of $w(g,h)$ is non-trivial in $\Gamma/N$. 
In particular, $w$ is not a law for $\Gamma/N$. 
But $\Gamma/N$ is also a quotient of $\Gamma$, 
so $w$ is not a law for $\Gamma$ either. 
\end{proof}

Theorem \ref{GSmainthm} is immediate from 
Corollary \ref{GSlawlesscoroll} and the next result. 
Recall that $\pi_{\Gamma} ^S$ is the \emph{torsion growth function} of $\Gamma$ 
(see Example \ref{TorsionEx} above). 

\begin{propn} \label{GSTachPropn}
For all $k \geq 2$ there is a torsion Golod-Shafarevich $p$-group 
$\Gamma = \langle X_k \mid R' \rangle$, 
and a constant $C > 0$ such that for all $n \in \mathbb{N}$, 
\begin{equation}
\pi_{\Gamma} ^{\pi(X_k)} (n) \leq C n \text{.}
\end{equation}
\end{propn}

\begin{proof}[Proof of Theorem \ref{GSmainthm}]
Let $v$ be as in Proposition \ref{GSTachPropn}. 
Let $W = \lbrace x^{p^k} : k\in\mathbb{N} \rbrace$. 
Then for all $n \in \mathbb{N}$, 
\begin{equation*}
\mathcal{A}_{\Gamma} ^{X_k} (Cn) 
\geq \mathcal{A}_{\Gamma,W} ^{X_k} (Cn) \geq n+1
\end{equation*}
(by Lemma \ref{supersetlawsineq} and Example \ref{TorsionEx}). 
\end{proof}

\begin{proof}[Proof of Proposition \ref{GSTachPropn}]
Let $w_1 , w_2 , w_3 \ldots$ be an enumeration of the non-trivial 
elements of $F(X_k)$, ordered such that $\lvert w_n \rvert$ 
is non-decreasing. 
Recall that there exists $C > 0$ such that 
$\lvert B_{X_k} (l) \rvert \leq C (2k - 1)^l$. 
Choose $q > 1$ and let $0 < c < \log(q)/\log (2k-1)$, 
so that: 
\begin{equation}
a_m := C q^{p^m} \geq \big\lvert B_{X_k} (c p^m) \big\rvert\text{.}
\end{equation} 
for $m \geq 1$ (with $a_0 := 0$). 
Choose $m_0 \in \mathbb{N}$ and set $r_n = w_n ^{p^{m+m_0}}$ 
for $a_{m-1} + 1 \leq n \leq a_m$ 
and $R^{\prime} = \lbrace r_n : n \in \mathbb{N} \rbrace$. 
Let $\Gamma = \langle X_k | R^{\prime} \rangle$ 
and let $g \in \Gamma$. 
If $m \in \mathbb{N}$ is such that 
$c p^{m-1} \leq \lvert g \rvert_{X_k} \leq c p^m$, 
then there exists $1 \leq n \leq a_m$ 
such that $g = w_n$ in $\Gamma$, 
so that the order of $g$ in $\Gamma$ divides 
$p^{m+m_0} \leq p^{m_0 +1} \lvert g \rvert_{X_k} / c$, 
so $\tau_{\Gamma} ^{X_k}$ grows at most linearly. 

It therefore suffices to check that 
we can choose $q$ and $m_0$ such that $\Gamma$ is Golod-Shafarevich. 
By Lemma \ref{degpowerlem}, 
$D(r_n) \geq p^{m+1}$ for $n > a_m$, 
hence for $\tau \in (0,1)$, 
\begin{equation} 
\label{GSreldefineq}
\sum_{r \in R^{\prime}}\tau^{D(r)}\leq\sum_{m=1} ^{\infty} a_m \tau ^{p^{m+m_0}} 
= C \sum_{m=1} ^{\infty} \Big( q\tau^{p^{m_0}} \Big)^{p^m}  
= C \sum_{m=1} ^{\infty} h^{p^m}  
\end{equation}
where $h = q\tau^{p^{m_0}}$. 
It is clear that the right-hand side of (\ref{GSreldefineq}) 
can be made arbitrarily small by making $h$ arbitrarily small. 
If we take $q=2$ and $\tau = 3/4$, 
then we have (\ref{GSineq}) provided $h$ is sufficiently small that 
the right-hand side of (\ref{GSreldefineq}) is $< (3k-4)/4$. 
This can be achieved for $m_0$ larger than an absolute constant. 
\end{proof}

\begin{rmrk}
\normalfont
It is not difficult to strengthen the proof of Proposition \ref{GSTachPropn} 
to show that if $\tilde{\Gamma}$ is Golod-Shafarevich, 
then there is a torsion Golod-Shafarevich $p$-group $\Gamma$, 
a surjective homomorphism 
$\pi : \tilde{\Gamma} \rightarrow \Gamma$ 
and a constant $C > 0$
such that for all $n \in \mathbb{N}$, 
\begin{equation}
\pi_{\Gamma} ^{\pi(S)} (n) \leq C n \text{.}
\end{equation}
For, the pro-$p$ completion of $\tilde{\Gamma}$ already has a presentation 
satisfying (\ref{GSineq}). 
Adding in the relations $R'$ from the proof of Proposition \ref{GSTachPropn} 
yields a presentation for the pro-$p$ completion of $\Gamma$, 
which has at most linear torsion growth. 
But we can make the contribution of $R'$ to the left-hand side of (\ref{GSineq}) 
arbitrarily small, so that the inequality (\ref{GSineq}) still holds, 
and $\Gamma$ is still Golod-Shafarevich. 
\end{rmrk}

\section{Grigorchuk's group} \label{BranchSect}

For the duration of this Section, 
$\Gamma$ will denote the first Grigorchuk group. 
For background on this group and automorphisms of rooted trees more generally, 
we refer to \cite{GrigChap}. 
The group $\Gamma$ is defined as a group of automorphisms of the binary 
rooted tree $\mathcal{T}$: 
$V(\mathcal{T}) = \lbrace 0,1 \rbrace ^{\ast}$ is 
the set of finite formal words in the alphabet $\lbrace 0,1 \rbrace$, 
and for each $v \in \lbrace 0,1 \rbrace ^{\ast}$, $\epsilon \in \lbrace 0,1\rbrace$, 
there is an edge joining $v$ to $v\epsilon$. 
For each $v \in V(\mathcal{T})$, 
let $\mathcal{T}_v$ be the subtree rooted at $v$, that is 
the induced subgraph on $\lbrace vw : w \in V(\mathcal{T}) \rbrace$. 
Note that $\mathcal{T}_v \cong \mathcal{T}$ via $vw \mapsto w$. 
Let: 
\begin{center}
$\rst (v) = \lbrace g \in \Aut (\mathcal{T}) 
: w^g=w \text{ for all }w \in V(\mathcal{T}) \setminus V(\mathcal{T}_v)\rbrace$ 
\end{center}
be the \emph{restricted stabilizer} at $v$. 
Then for every $v$, 
the above isomorphism $\mathcal{T}_v \cong \mathcal{T}$ 
induces an isomorphism $\rst (v) \rightarrow \Aut (\mathcal{T})$; 
we write $g\mid_{\mathcal{T}_v} \in \Aut (\mathcal{T})$ 
for the image under this isomorphism of $g \in \rst(v)$. 

Let $V_n =  \lbrace 0,1 \rbrace^n$ be the set of words of length $n$ 
(geometrically, the set of vertices at distance $n$ from the root of the tree). 
Let $\Stab (n) \leq \Aut(\mathcal{T})$ be the pointwise stabilizer of $V_n$. 
Then for $g \in \Stab (n)$, there exist unique 
$g_v \in \rst (v)$, for $v \in V_n$, such that: 
\begin{equation} \label{rstdecompeqn}
g = \prod_{v \in V_n} g_v
\end{equation}
(note that the $g_v$, being disjointly supported, commute). 
Extending our notation above, we write $g\mid_{\mathcal{T}_v}=g_v \mid_{\mathcal{T}_v}$ 
for $g \in \Stab(n)$, $v \in V_n$ and $g_v \in \rst (v)$ as in (\ref{rstdecompeqn}). 
The decomposition (\ref{rstdecompeqn}) yields an isomorphism 
$\Aut(\mathcal{T})^{V_n} \cong \Stab(n)$, 
and for $K \leq \Aut(\mathcal{T})$ we shall identify $K^{V_n}$ 
with its image in $\Aut(\mathcal{T})$ under this isomorphism. 
Likewise, when $g$ and $g_v$ are as in (\ref{rstdecompeqn}) 
we shall write $g = (g_v)_{v \in V_n}$. 

The automorphisms $a$, $b$, $c$ and $d$ of $\mathcal{T}$ 
are defined as follows: 
\begin{center}
$\begin{array}{ll}
(0 w)^a = 1 w; & (1 w)^a = 0 w; \\
(0 w)^b = 0 (w^a); & (1 w)^b = 1 (w^c); \\
(0 w)^c = 0 (w^a); & (1 w)^c = 1 (w^d); \\
(0 w)^d = 0 w; & (1 w)^d = 1 (w^b)
\end{array}$
\end{center}
for any $w \in \lbrace 0,1 \rbrace^{\ast}$. 
In other words, $a$ swaps the two subtrees $\mathcal{T}_0$ and $\mathcal{T}_1$, 
and $b,c,d \in \Stab (1)$ are given by: 
\begin{center}
$b=(a,c)$, $c=(a,d)$, $d=(1,b)$. 
\end{center}
Grigorchuk's group $\Gamma$ is defined to be 
the subgroup of $\Aut(\mathcal{T})$ generated by $S = \lbrace a,b,c,d \rbrace$. 
Let $x = abab \in \Gamma$ and let $K = \langle x \rangle ^{\Gamma} 
\vartriangleleft \Gamma$ be the normal closure of $x$ in $\Gamma$. 

\begin{propn}[\cite{GrigChap} Proposition 8]
We have $K^{V_1} \leq K$ (and hence $K^{V_n} \leq K$ for all $n$). 
\end{propn}

\begin{propn} \label{RstLengthProp}
Let $y_n \in K^{V_n}$ be given by: 
\begin{equation*}
y_n \mid_{\mathcal{T}_v} = \Big\{ \begin{array}{cc} x & v=0^n \\
e & \text{otherwise}
\end{array}
\end{equation*}
for $v \in V_n$. 
Then there exists $C>0$ such that for all $n$, 
$\lvert y_n \rvert_S \leq C (1+\sqrt{3})^n$. 
\end{propn}

\begin{proof}
We have $y_0 = x$ and $y_1 = y = (x,1)$. 
As noted in the proof of \cite[Proposition 9]{GrigChap}, 
$[x,y] = (x^{-1},1,1,1) = y_2 ^{-1}$, so $y_2 = [y_1,y_0]$. 
Since for $n \geq 1$ we have $y_n = (y_{n-1},1)$, 
so by induction, for $n \geq 3$ we have: 
\begin{equation} \label{GrigCommEqn}
y_n =(y_{n-1},1)=\big([y_{n-2},y_{n-3}],1\big)
=\big[(y_{n-2},1),(y_{n-3},1) \big] =[y_{n-1},y_{n-2}]\text{.}
\end{equation} 
Now there exists $C_0 > 0$ such that 
$\lvert y_0 \rvert_S , \lvert y_1 \rvert_S \leq C_0$ 
(the latter since $y_1 \in K^{V_1} \leq K \leq \Gamma$), 
and by (\ref{GrigCommEqn}) we have 
$\lvert y_n \rvert_S \leq 2 \lvert y_{n-1} \rvert_S + \lvert y_{n-2} \rvert_S$ 
for $n \geq 2$. 
Solving the corresponding recurrence yields the required bound. 
\end{proof}


Let $W_n$ be the $(n+1)$-fold iterated regular wreath product of $C_2$; 
that is $W_0 = C_2$ and $W_{n+1} = W_n \wr C_2$. 
Alternatively we may view $W_n$ as a subgroup of $\Aut (\mathcal{T})$, 
as follows: let $a \in \Aut (\mathcal{T})$ be as above, 
and identify $W_0 = C_2$ with $\langle a \rangle \leq \Aut (\mathcal{T})$. 
Having defined $W_n \leq \Aut (\mathcal{T})$, 
a general element of 
$W_{n+1}$ is $(g_0,g_1) a^{\epsilon} \in W_n \wr C_2$ 
(with $g_0 , g_1 \in W_n$ and $\epsilon \in \lbrace 0,1 \rbrace$). 
We identify this with the unique $g \in \Aut (\mathcal{T})$ 
satisfying: 
\begin{center}
$g a^{\epsilon} \in \Stab(1)$ and $(g a^{\epsilon})\mid_{\mathcal{T}_v} = g_v$ for $v\in V_1$. 
\end{center}
It is easily seen that this identification yields an embedding 
of $W_{n+1}$ as a subgroup of $\Aut (\mathcal{T})$. 
Moreover, the action of $W_n \leq \Aut (\mathcal{T})$ on $V_{n+1}$ 
is faithful and yields an isomorphism of permutation groups 
from $W_{n+1}$ to the imprimitive permutational wreath product 
 $\langle a \rangle \wr_{V_{n+1}} W_n$ 
(since $V_{n+2} = V_{n+1} \times V_1$ via the identification of $v\epsilon$ 
with $(v,\epsilon)$). 

Define, for $n \in \mathbb{N}$, $a_n \in \Aut (\mathcal{T})$ by: 
$a_0 = a$ and $a_{n+1} \in \rst \big( 0^{n+1} \big)$ by: 
\begin{equation*}
a_{n+1} |_{T_{0^{n+1}}} = a
\end{equation*} 
so that $S_n = \lbrace a_0 , a_1 , \ldots , a_n \rbrace$ generates $W_n$. 

\begin{lem} \label{WPSchreierLem}
For any $v \in V_{n+1}$, there exists $h \in W_n$ 
with $\lvert h \rvert_{S_n} \leq n+1$ 
and $ (0^{n+1})^{h} = v$. 
\end{lem}

\begin{proof}
If $v = \epsilon_0 \epsilon_1 \cdots \epsilon_n$, 
then $h = a_n ^{\epsilon_n} a_{n-1} ^{\epsilon_{n-1}} 
\cdots a_1 ^{\epsilon_1} a_0 ^{\epsilon_0}$ works. 
\end{proof}

\begin{propn} \label{WPlawlengthprop}
$W_n$ has no law in $F_k$ of length at most $n+1$. 
In other words, 
$B(n+1) \setminus \lbrace 1 \rbrace \subseteq N_k (W_n)$. 
More precisely, for any $1 \neq w \in F_k$ with $\lvert w \rvert \leq n+1$, 
there exist $g_1 , \ldots , g_k \in W_n$ with: 
\begin{equation} \label{WPlawlengtheqn}
w(g_1 , \ldots , g_k) \neq e
\text{ and } \sum_{i=1} ^k \lvert g_i \rvert_{S_n} \leq (n+1)^2. 
\end{equation}
\end{propn}

\begin{proof}
We proceed by induction on $n$. 
In fact, we make a stronger claim. 
For $w \in F_k$ and $0 \leq m \leq \lvert w \rvert$ 
define $w^{(m)}$ to be the \emph{$m$-prefix} of $w$, 
that is, $w^{(m)} \in F_k$ is the unique element satisfying 
(i) $\lvert w^{(m)} \rvert = m$ and 
(ii) there exists $u^{(m)} \in F_k$ 
such that $\lvert u^{(m)} \rvert =\lvert w \rvert - m$ 
and $w = w^{(m)} u^{(m)}$ (so that $w^{(0)}$ is the empty word 
and $w^{(\lvert w \rvert)} = w$). 
Our claim is that for all $w \in F_k$ 
with $\lvert w \rvert \leq n+1$, 
there exists $\mathbf{g} = (g_1 , \ldots , g_k) \in W_n ^k$ with: 
\begin{equation*}
\sum_{i=1} ^k \lvert g_i \rvert_{S_n} \leq (n+1)^2
\end{equation*}
such that the points: 
\begin{center}
$v_0 = 0^{n+1}, 
v_1 = v_0 ^{w^{(1)}(\mathbf{g})},
\ldots, v_0 ^{w^{(\lvert w \rvert)}(\mathbf{g})} 
=v_0 ^{w(\mathbf{g})}$ 
\end{center}
are all distinct. 
In particular, $v_0 \neq v_0 ^{w(\mathbf{g})}$, 
so $w(\mathbf{g}) \neq e$, and (\ref{WPlawlengtheqn}) follows. 
The claim clearly holds for $n=0$. 

Let $n \geq 1$, let $1 \neq w \in F_k$ with $2 \leq \lvert w \rvert \leq n+1$, 
(if $\lvert w \rvert = 1$ the claim is trivial) 
and suppose the claim fails for $w$ and $W_{n+1}$. 
Let $u = w^{(\lvert w \rvert-1)}$; WLOG 
(by permuting and inverting the variables $x_i$) 
we may assume that $w=ux_1$. 
Since $1 \neq u$, $\lvert u \rvert \leq n$, 
we may assume by induction that there exists 
$\mathbf{g}' = (g_1' , \ldots , g_k') \in W_{n-1} ^k$ 
witnessing the truth of the claim for $u$ and $W_{n-1}$. 
Let $v_i ' = (0^n)^{u^{(i)}(\mathbf{g}')} \in V_n$ for $0 \leq i \leq \lvert u \rvert$ 
(a sequence of $\lvert u \rvert + 1$ distinct points). 
Now consider the $g_j '$ as elements of $W_n = \langle a \rangle \wr_{V_n} W_{n-1}$, 
acting naturally on $V_{n+1}$, and consider the points 
$\tilde{v}_i = (0^{n+1})^{w^{(i)}(\mathbf{g}')}$ 
for $0 \leq i \leq \lvert w \rvert = \lvert u \rvert + 1$. 
For $i \leq \lvert u \rvert$ we have $\tilde{v}_i = v_i ' 0$; 
these points are distinct. 
On the other hand, by assumption the $\tilde{v}_i$ are not all distinct, 
so we must have $\tilde{v}_{\lvert w \rvert} = \tilde{v}_0 = 0^{n+1}$. 

Let $h \in W_{n-1}$ with $\lvert h \rvert_{S_{n-1}} \leq n$ 
and $ (0^n)^{h} = v_{\lvert u \rvert} '$; 
such exists by Lemma \ref{WPSchreierLem}. 
Then: 
\begin{center}
$h^{-1} a_n h \in \rst (v_{\lvert u \rvert} ')$, 
with $(h^{-1} a_n h)\mid_{T_{v_{\lvert u \rvert} '}} = a$, 
and $\lvert h^{-1} a_n h \rvert_{S_n} \leq 2n+1$. 
\end{center}
Set $g_1 = h^{-1} a_n h g_1 '$ and $g_j = g_j '$ for $2 \leq j \leq k$. 
Consider $v_i = (0^{n+1})^{w^{(i)}(\mathbf{g})}$. 
We have $v_i = \tilde{v_i}$ for $0 \leq i \leq \lvert u \rvert$ 
since $v_i ' \neq v_{\lvert u \rvert} '$ for $i \leq \lvert u \rvert - 1$ 
(note that $w$ is a reduced word so the final letter of $u$ is not $x_1 ^{-1}$); 
these points are distinct. 
But then: 
\begin{center}
$v_{\lvert w \rvert} = v_{\lvert u \rvert} ^{g_1} 
 = ((v_{\lvert u \rvert} ' 0)^{h^{-1} a_n h})^{g_1 '} 
 = (v_{\lvert u \rvert} ' 1)^{g_1 '} 
 = 0^n 1$
\end{center}
differs in its final letter from all $v_i = v_i ' 0$ for $i \leq \lvert u \rvert$. 
Moreover, 
\begin{equation*}
\sum_{i=1} ^k \lvert g_i \rvert_{S_n} 
\leq (2n+1)+\sum_{i=1} ^k \lvert g_i ' \rvert_{S_n}.
\end{equation*}
The claim follows. 
\end{proof}

The next Proposition is essentially proved in the course of 
Proposition 10 of \cite{GrigChap}. 
We include a proof for the reader's convenience. 

\begin{propn} \label{WPEmbGrigProp}
For each $n \in \mathbb{N}$ there is an injective homomorphism 
$\Phi_n:W_n\rightarrow \Gamma$ sending $a_i$ 
to $k_{i+1} =y_{5i} ^4$ for $0\leq i\leq n$. 
\end{propn}

\begin{proof}
We follow the proof of Proposition 4 of \cite{GrigChap}, 
specialized to our setting. 
We have $k_{i+1} \in \rst (0^{5i})$, with $k_{i+1} \mid_{\mathcal{T}_{0^{5i}}} = x^4$, 
and $x^4 \in \Stab(4)$ is given, for $v\in V_3$, by: 
\begin{equation*}
x^4 \mid_{\mathcal{T}_{v0}} = a \text{ and } x^4 \mid_{\mathcal{T}_{v1}} = c; 
\end{equation*} 
in particular, every $k_{i+1}$ has order $2$. 
We prove the conclusion by induction on $n$. 
Certainly $W_0 = C_2 \cong \langle k_1 \rangle \leq \Gamma$, 
so the conclusion holds for $n=0$. 
Let $Q_n = \langle k_1 , \ldots , k_n \rangle \leq \Gamma$; 
we suppose by induction that there is an isomorphism $\Phi_n : W_{n-1}\rightarrow Q_n$ 
sending $a_i$ to $k_{i+1}$ for $0 \leq i \leq n-1$. 
Note that the natural isomorphism $\Aut(\mathcal{T}) \cong \rst (0^5)$ 
(induced by the isomorphism of trees $\mathcal{T} \cong \mathcal{T}_{0^5}$) 
sends $k_i$ to $k_{i+1}$; the restriction of this map to $Q_n$ yields 
an isomorphism from $Q_n$ to $P_n = \langle k_2 , \ldots , k_{n+1} \rangle$. 
Similarly the natural isomorphism $\Aut(\mathcal{T}) \cong \rst (0)$ 
sends $a_i$ to $a_{i+1}$, 
so restricts to an isomorphism from $W_{n-1}$ to $\langle a_1 , \ldots , a_n \rangle$. 
Composing, we have a monomorphism $\Psi_n : P_n \rightarrow W_n$ 
sending $k_{i+1}$ to $a_i$ for $1 \leq i \leq n$. 
We shall extend $\Psi_n$ from $P_n$ to $Q_{n+1}$. 

Since $P_n \leq \rst (0^5)$ and $k_1 \mid_{T_{0^4}} = a$, 
we have $P_n ^{k_1} \leq \rst (0^4 1)$, 
so $P_n , P_n ^{k_1} \leq \Gamma$ generate their direct product, 
and $k_1$ (being of order $2$) acts by conjugation by swapping the two factors. 
Thus $Q_{n+1} = \langle k_1 , P_n \rangle \cong P_n \wr C_2$, 
with the $C_2$-factor generated by $k_1$. 
To finish, we compose with the isomorphism 
$P_n \wr C_2 \cong W_{n-1} \wr C_2 = W_{n+1}$ induced by $P_n \cong W_{n-1}$. 
\end{proof}

One of the original motivations for introducing $\Gamma$ was the following, 
now-famous result. 

\begin{thm} \label{Grig2grpthm}
For all $g \in \Gamma$, 
there exists $k \in \mathbb{N}$ such that $g^{2^k} = 1$. 
\end{thm}

As we have seen (Example \ref{TorsionEx} above) 
slow torsion growth yields fast lawlessness growth for $p$-groups. 
The following is a consequence of Theorem 7.7 of \cite{BartSuni} 
(see also the bullet-points at the end of Section 1 of that paper). 

\begin{thm} \label{Grigsmallorderthm}
There exists $C>0$ such that for all $g \in \Gamma$, 
$o(g) \leq C \lvert g \rvert_S ^{3/2}$. 
\end{thm}

\begin{proof}[Proof of Theorem \ref{Grigmainthm}]
For the upper bound, let $n \in \mathbb{N}$ 
and let $1 \neq w \in F_k$ with $\lvert w \rvert \leq n+1$. 
Let $g_1 , \ldots , g_k \in W_n$ be as in Proposition \ref{WPlawlengthprop} 
and let $\Phi_n : W_n \rightarrow \Gamma$ be as in Proposition \ref{WPEmbGrigProp}. 
We have: 
\begin{equation*}
e \neq \Phi_n \big( w (g_1 , \ldots , g_k) \big) 
= w \big( \Phi_n (g_1) , \ldots , \Phi_n (g_k)  \big)
\end{equation*}
so that: 
\begin{equation*}
\chi_{\Gamma} ^S (w) \leq \sum_{i=1} ^k \lvert \Phi_n (g_i) \rvert_S
\leq 4 C (1+\sqrt{3})^{5n} \sum_{i=1} ^k \lvert g_i \rvert_{S_n}
\leq 4 C (n+1)^2 (1+\sqrt{3})^{5n}
\end{equation*}
(the second inequality being by Proposition \ref{RstLengthProp}). Thus: 
\begin{equation*}
\mathcal{A}_{\Gamma} ^S (n) \leq 4 C (n+1)^2 (1+\sqrt{3})^{5n} \preceq \exp (n). 
\end{equation*}
For the lower bound, let $w_m (x,y)=x^{2^m} \in F_2$. 
Suppose that $g,h \in \Gamma$ satisfies 
$\lvert g \rvert_S , \lvert h \rvert_S \leq l$. 
Then by Theorems \ref{Grig2grpthm} and \ref{Grigsmallorderthm}, 
$w_m (g,h) = e$ for all $m \leq \log_2 C + (3/2) \log_2 l$. Thus: 
\begin{center}
$\mathcal{A}_{\Gamma} ^S (2^m) 
\geq \chi_{\Gamma} ^S (w_m) \geq l +1
\geq 2^{2m/3}/C$. 
\end{center}
\end{proof}

\section{Thompson's group $\mathbf{F}$} \label{ThompSect}

In this Section we prove Theorem \ref{ThompMainThm}. 
We adopt the following model for Thompson's group $\mathbf{F}$. 

\begin{defn}
$\mathbf{F}$ is the subgroup of $\Homeo (\mathbb{R})$ 
generated by the homeomorphisms $A$ and $B$, where: 
\begin{equation*}
A(x)=x+1\text{ and }
B(x)=\left\{
  \begin{array}{@{}ll@{}}
    x & x\leq 0 \\
    2x & x \in (0,1] \\
    x+1 & x>1 
  \end{array}\right. 
\end{equation*} 
\end{defn}

\begin{propn} \label{ThompCharProp}
$\mathbf{F}$ is the group of orientation-preserving 
piecewise-linear homeomorphisms $H$ of $\mathbb{R}$ 
which are differentiable except at finitely many dyadic rational 
numbers; all of whose slopes are powers of $2$, 
and such that there exist integers $K$ and $L$ such that 
$H(x)=x+K$ for all $x \in \mathbb{R}$ sufficiently large 
and $H(x)=x+L$ for all $x \in \mathbb{R}$ sufficiently small.  
\end{propn}

Our proof of Theorem \ref{ThompMainThm} is based closely 
on \cite{BrinSqui}, where the following was proved. 

\begin{thm}[Brin--Squier]
The group $\mathbf{F}$ is lawless. 
\end{thm}

The argument presented in \cite{BrinSqui} 
is constructive, 
and we extract from it a linear upper bound on 
$\mathcal{A}_{\mathbf{F}}$, as follows. 
Define $T : \mathbb{R} \rightarrow \mathbb{R}$ 
on each interval $[8n,8(n+1)]$ by: 
\begin{equation*}
T(x) = \left\{
  \begin{array}{@{}ll@{}}
    2 (x-4n) & x \in [8n,8n+1] \\
    x+1 & x \in [8n+1,8n+2] \\
    (x+(8n+4))/2 & x \in [8n+2,8n+6] \\
    x-1 & x \in [8n+6,8n+7] \\
    2(x-4(n+1)) & x \in [8n+7,8(n+1)] 
  \end{array}\right. 
\end{equation*}
(so that $T$ restricts to an orientation-preserving 
homeomorphism of $[8n,8(n+1)]$)
and set $U = T^2$. 
For $n\in\mathbb{N}$ define $U_n :\mathbb{R}\rightarrow\mathbb{R}$ by: 
\begin{equation*}
U_n (x) = \left\{
  \begin{array}{@{}ll@{}}
    U(x) & 0 \leq x \leq 8(n+1) \\
    x & \text{otherwise} 
  \end{array}\right. 
\end{equation*}
and $V_n = A^2 U_n A^{-2}$, so that $U_n$ and $V_n \in \mathbf{F}$. 

\begin{lem} \label{ThompLengthLem}
There exists $C > 0$ such that for all $n \in \mathbb{N}$, 
\begin{center}
$\big\lvert U_n \big\rvert_S , \big\lvert V_n \big\rvert_S \leq C (n+1)$. 
\end{center}
\end{lem}

\begin{proof}
By Proposition \ref{ThompCharProp}, $U_0 \in \mathbf{F}$. 
Let $M = \lvert U_0 \rvert_S \in \mathbb{N}$. 
Then $U_{n+1} = A^8 U_n A^{-8} U_0$, 
so $\lvert U_n \rvert_S \leq (M+16)n + M$ 
and $\lvert V_n \rvert_S \leq (M+16)n + M+4$. 
\end{proof}

\begin{propn} \label{Thomp}
Let $w \in F_2$ be non-trivial, with $\lvert w \rvert \leq n$. 
Then $w (U_n,V_n) \neq e$. 
\end{propn}

\begin{proof}
This is essentially the content of \cite{BrinSqui} Section 4; 
we give a sketch. 
In \cite{BrinSqui} there are defined homeomorphisms $f_0,f_1$ of $S^1$, 
such that $f_0 ^2, f_1 ^2$ freely generate a rank-$2$ 
free subgroup of $\Homeo (S^1)$. 
This is shown using a ping-pong argument. 
Identifying $S^1$ with $\mathbb{R}/8 \mathbb{Z}$ and 
letting $\pi : \mathbb{R} \rightarrow S^1$ be the induced covering map, 
it follows that for any continuous lifts $\tilde{f}_0, \tilde{f}_1$ 
of $f_0,f_1$ to this cover, 
$\tilde{f}_0 ^2, \tilde{f}_1 ^2$ freely generate a rank-$2$ 
free subgroup of $\Homeo (\mathbb{R})$. 
As per the description of $f_0$ given in \cite{BrinSqui}, 
our map $T$ is such a lift of $f_0$ to $\mathbb{R}$ fixing $0$. 

It is then shown that similarly, if $\tilde{f}_0$ is a continuous lift 
of $f_0$ to $\mathbb{R}$ fixing $0$; 
and $g_0 : \mathbb{R} \rightarrow \mathbb{R}$ is given by: 
\begin{equation*}
g_0 (t) = \left\{
  \begin{array}{@{}ll@{}}
    \tilde{f}_0 ^2 (t) & t \in [-4(n+1),4(n+1)] \\
    t & \text{otherwise}
  \end{array}\right. 
\end{equation*}
(so that $g_0 \in \Homeo (\mathbb{R})$), 
and $g_1 = A^2 g A^{-2}$, 
then $w(g_0,g_1)\neq e$ in $\Homeo (\mathbb{R})$
($T_2$ is the notation of \cite{BrinSqui} is $A$ in our notation 
and actions in \cite{BrinSqui} are on the right while 
ours are on the left). 
Taking $\tilde{f}_0 =T$ as above, we have $U_n =A^{4(n+1)} g_0A^{-4(n+1)}$ 
and $V_n = A^{4(n+1)} g_1A^{-4(n+1)}$, 
so $w(U_n,V_n) = A^{4(n+1)} w(g_0,g_1) A^{-4(n+1)} \neq e$ also. 
\end{proof}

\begin{proof}[Proof of Theorem \ref{ThompMainThm}]
By Lemma \ref{ThompLengthLem} and Proposition \ref{Thomp}, 
any non-trivial $w \in F_2$, with $\lvert w \rvert \leq n$, 
satisfies $\chi_{\mathbf{F}} ^{\lbrace A,B \rbrace} (w) \leq 2C(n+1)$. 
\end{proof}

\section{Residual finiteness growth}\label{RFSect}

Let $\mathcal{C}$ be a class of finite groups. 
Recall that $\Gamma$ is \emph{residually $\mathcal{C}$} if, 
for every $1 \neq g \in \Gamma$, 
there exists $Q \in \mathcal{C}$ and a surjective 
homomorphism $\pi : \Gamma \rightarrow Q$ such that $\pi(g) \neq 1$. 
In this case we denote by $D_{\Gamma,\mathcal{C}} (g)$ 
the minimal value of $\lvert Q \rvert$ among $Q \in \mathcal{C}$ 
admitting such a homomorphism $\pi$. 
If $S$ is a finite generating set for $\Gamma$, 
then the \emph{residual $\mathcal{C}$-finiteness growth function} 
of $\Gamma$ (with respect to $S$) is: 
\begin{center}
$\mathcal{F}_{\Gamma,\mathcal{C}} ^S (n) 
= \max \lbrace D_{\Gamma,\mathcal{C}} (g) : \lvert g \rvert_S \leq n \rbrace$. 
\end{center}
Residual finiteness growth was introduced by Bou-Rabee \cite{BouRab}
and has been extensively studied for a wide variety of 
residually finite groups. 
Two classes $\mathcal{C}$ which are of particular interest 
are the class of \emph{all} finite groups, and the class 
of all \emph{$p$-groups} (for $p$ a fixed prime). 
In these cases we denote the function 
$\mathcal{F}_{\Gamma,\mathcal{C}} ^S$ by 
$\mathcal{F}_{\Gamma} ^S$ 
and $\mathcal{F}_{\Gamma,p} ^S$, 
respectively. 

\begin{propn} \label{LawsRFprop}
Let $W \subseteq N_2 (\Gamma)$. 
Let $f : \mathbb{N} \rightarrow \mathbb{N}$ 
and suppose that for all $l$, 
there exists $w_l \in W$ of length at most $f(l)$ 
which is a law for every member of $\mathcal{C}$ 
of order at most $l$. Then: 
\begin{equation}
\mathcal{F}_{\Gamma,\mathcal{C}} ^S 
\big(f(l)\cdot (\mathcal{A}_{\Gamma,W} ^S \circ f)(l) \big) 
> l \text{.}
\end{equation}
\end{propn}

\begin{proof}
Let $w_l \in W$ be as in the statement. 
There exist $g,h \in \Gamma$ such that: 
\begin{equation*}
\max \lbrace\lvert g \rvert_S ,\lvert h \rvert_S \rbrace 
\leq \lvert g \rvert_S + \lvert h \rvert_S
\leq \mathcal{A}_{\Gamma,W} ^S \big( f(l) \big)
\end{equation*}
and $w_l (g,h) \neq e$, while by construction, 
$D_{\Gamma,\mathcal{C}} \big( w_l (g,h) \big) > l$ and: 
\begin{equation*}
\lvert w_l (g,h)\rvert_S 
\leq \lvert w_l\rvert\cdot
\max\lbrace\lvert g\rvert_S ,\lvert h\rvert_S\rbrace 
\leq f(l)\cdot\mathcal{A}_{\Gamma,W} ^S \big( f(l) \big)\text{.}
\end{equation*}
\end{proof}

Proposition \ref{LawsRFprop} is useful in the presence 
of good upper bounds on the function $f$. 
The best currently known bounds for the classes 
of finite groups and finite $p$-groups are as follows, 
taken from \cite{BradThom} and \cite{Elka}, respectively. 

\begin{thm} \label{BThomThm}
For each $n \in \mathbb{N}$ there exists $1 \neq w_n \in F_2$ 
such that: 
\begin{itemize}
\item[(i)] For all finite groups $G$ of order at most $n$, 
$w_n$ is a law for $G$; 
\item[(ii)] For all $\delta > 0$, 
$\lvert w_n \rvert = O_{\delta} (n^{2/3} \log(n)^{3+\delta})$. 
\end{itemize}
\end{thm}

\begin{thm} \label{ElkaThm}
For each $m \in \mathbb{N}$ there exists $1\neq w_m ^{\prime}\in F_2$ 
such that: 
\begin{itemize}
\item[(i)] For all nilpotent groups $G$ of class at most $m$, 
$w_m ^{\prime}$ is a law for $G$ 
(in particular, $w_m ^{\prime}$ is a law for every finite $p$-group 
of order at most $p^m$); 
\item[(ii)] $\lvert w_m ^{\prime} \rvert = O (m^{\alpha})$, 
where $\alpha = \log(2)/(\log(1+\sqrt{5})-\log(2)) \approx 1.440$. 
\end{itemize}
\end{thm}

\begin{coroll} \label{AllgrpsCoroll}
Let $\Gamma$ be a finitely generated lawless group 
and let $\delta > 0$. Then: 
\begin{equation*}
\mathcal{F}_{\Gamma} \big( l \mathcal{A}_{\Gamma} (l) \big) 
\succeq l^{3/2} / \log (l)^{9/2 + \delta}\text{.}
\end{equation*}
\end{coroll}

\begin{proof}
Let $w_n$ be as in Theorem \ref{BThomThm} 
and set $W = \lbrace w_n \rbrace_{n \in \mathbb{N}}$. 
Apply Proposition \ref{LawsRFprop} with $\mathcal{C}$ 
the class of all finite groups. 
\end{proof}

\begin{coroll}
Let $\Gamma$ be a lawless group and let $p$ be a prime. Then: 
\begin{equation*}
\mathcal{F}_{\Gamma,p} \big( l \mathcal{A}_{\Gamma} (l) \big) 
\succeq \exp \big( \log (p) l^{1/\alpha} \big),
\end{equation*}
where $\alpha$ is as in Theorem \ref{ElkaThm} 
(so that $1/\alpha \approx 0.694$). 
\end{coroll}

\begin{proof}
Let $w_m ^{\prime}$ be as in Theorem \ref{ElkaThm}. 
and set $W = \lbrace w_m ^{\prime} \rbrace_{m \in \mathbb{N}}$. 
Apply Proposition \ref{LawsRFprop} with $\mathcal{C}$ 
the class of all finite $p$-groups. 
\end{proof}

\section{Equations with coefficients} \label{CoeffSect}

Laws are generalized by \emph{mixed identities}. 
Given a group $\Gamma$, a non-trivial element $w$ of the 
free product $\Gamma \ast F_k$, 
lying in the kernel of every homomorphism 
$\Gamma \ast F_k \rightarrow \Gamma$ 
which restricts to the identity on $\Gamma$. 
A mixed identity is therefore a law for $\Gamma$ precisely when it lies in $F_k$. 
Mixed-identity-free (MIF) groups are rather rarer than lawless groups: 
any group decomposing as a non-trivial direct product, 
or with a non-trivial finite conjugacy class satisfies a mixed identity, 
as does any group possessing a non-trivial normal subgroup 
with a mixed identity. 

One may define an analogue of lawlessness growth for mixed identities. 
For $\Gamma$ a finitely generated group, 
and $w \in \Gamma \ast F_k$ non-trivial, not a mixed identity for $\Gamma$, 
the complexity of $w$ in $\Gamma$ is define exactly as for elements of $F_k$: 
it is the minimal word-length of a $k$-tuple $(g_i)$ in $\Gamma$ such that 
$w$ does not lie in the kernel of the homomorphism 
$\Gamma \ast F_k \rightarrow \Gamma$ 
restricting to the identity on $\Gamma$ and sending $x_i$ to $g_i$. 
To define the induced MIF growth function 
$\mathcal{M}_{\Gamma}$ of $\Gamma$, 
we take the maximal complexity over elements $w \in \Gamma \ast F_k$ 
of word-length at most $l$. 
As with lawlessness growth, it is easy to see that 
the equivalence class of the function $\mathcal{M}_{\Gamma}$ 
does not depend on a choice of finite generating set for $\Gamma$ 
or $\Gamma \ast F_k$. 
Moreover, since $\Gamma \ast F_k$ embeds into $F_k \ast \mathbb{Z}$ 
for every $k$, we may assume $k=1$. 
Henceforth we take $\mathbb{Z} = \langle x \rangle$ 
and fix a finite generating set $S$ for $\Gamma$.  

\begin{rmrk}
\normalfont
It might alternatively occur to one to assign length one to all 
coefficients from $\Gamma$ appearing in $w$. 
The distinction between such a length function and the 
word metric on $\Gamma \ast F_k$ coming from a finite generating set 
is roughly the same as that 
between the degree and ``height'' of a polynomial with integer coefficients. 
Importantly though, if we were to assign length one to every element of $\Gamma$, 
then for $\Gamma$ infinite, 
there would be infinitely many words of bounded length, 
so it would not be clear that the associated MIF growth function 
would take finite values, even for $\Gamma$ MIF-free. 
\end{rmrk}

We recall the following basic fact about subgroups of free groups, 
the proof of which is an easy consequence of the uniqueness 
of reduced-word representatives for elements of free products. 

\begin{lem} \label{freeprodlem}
Let $\Gamma$ and $\Delta$ be nontrivial groups, and let $e\neq h \in \Delta$. 
Then $\lbrace [g,h] : e \neq g \in \Gamma \rbrace$ 
freely generate a free subgroup of $\Gamma \ast \Delta$. 
\end{lem}

\begin{proof}[Proof of Theorem \ref{MIFLBThm}]
Let $g_1 , \ldots , g_m$ be an enumeration of the nontrivial elements 
of $B_S (l)$, so that for some $C>0$ (independent of $l$) we have $m \leq \exp(Cl)$. 
For $1 \leq i \leq m$ let $w_i = [g_i , x]$. 
By Lemma \ref{freeprodlem} the $w_i$ freely generate a free subgroup $F$ of 
$\Gamma \ast \mathbb{Z}$ of rank $m$. 
Applying the construction of Proposition \ref{KozThoProp} to the $w_i$, 
we obtain a nontrivial word $w$ in the $w_i$ of 
reduced length at most $16 m^2$, with 
the property that, whenever $\pi:F \rightarrow \Gamma$ 
is a homomorphism whose kernel contains some $w_i$, 
$w \in \ker (\pi)$ also. 

Since every $y_i$ has word-length at most $2(l+1)$ in $S \cup \lbrace x \rbrace$, 
$w$ has word-length at most $32 (l+1) m^2 \preceq \exp(l)$ 
when viewed as an element of $\Gamma \ast \mathbb{Z}$. 
Meanwhile, $w$ has complexity at least $l+1$ in $\Gamma \ast \mathbb{Z}$. 
For every homomorphism $\Gamma \ast \mathbb{Z} \rightarrow \Gamma$ 
resticting to the identity on $\Gamma$ restricts to a homomorphism 
$F \rightarrow \Gamma$. 
If such a homomorphism sends $x$ to some $g_i$, 
then $w_i$ lies in its kernel, and by Proposition \ref{KozThoProp}, 
so does $w$. 
\end{proof}

Now we turn to the proof of Theorem \ref{MIFfreeThm}. 
Henceforth $\Gamma$ is a finite-rank nonabelian free group with free basis $S$.

\begin{propn} \label{MIFfreePropn}
Let: 
\begin{equation} \label{freeprodcycredform}
w (x) = a_1 x^{k_1} \cdots a_l x ^{k_l} \in \Gamma \ast \mathbb{Z}
\end{equation}
for some $e \neq a_i \in \Gamma$ and $k_i \in \mathbb{Z}$ 
with $k_i \neq 0$ for $i \leq l-1$.  
Suppose $u \in \Gamma$ is 
not a proper power in $\Gamma$, 
and that $[u,a_i] \neq e$ for all $i$. 
Then for all $m \in \mathbb{N}$ sufficiently large that: 
\begin{equation}
\lvert u^m \rvert_S \geq \lvert u^2 \rvert_S l + \sum_{i=1} ^l \lvert a_i \rvert_S,
\end{equation} 
$w(u^m)$ does not commute with $u$ 
(and in particular $w(u^m)\neq e$). 
\end{propn}

\begin{proof}
We proceed by induction on $l$. 
For $l=1$ the claim is clear for any $m \geq 1$: 
$a_1 u^{k_1 m}$ commutes with $u$ iff $a_1$ does. 
Let $l \geq 2$ and suppose the claim holds for smaller $l$. 
If $w(u^m)$ commutes with $u$, then they are powers of a common word. 
Since $u$ is not a proper power, there exists $k_0 \in \mathbb{Z}$ such that 
$u^{k_0} a_1 u^{k_1 m} \cdots a_l u ^{k_l m} = e$. 

Consider the Cayley graph $\mathcal{C}$ of $\Gamma$ 
with respect to the free basis $S$. 
Since $\mathcal{C}$ is a tree there is, for every $g,h \in \Gamma$, 
a unique reduced path $[g,h]$ from $g$ to $h$ in $\mathcal{C}$ 
(of length equal to $\lvert g^{-1} h \rvert_S$). 
Writing $g_0=e$, $h_0=u^{k_0}$ and $g_i = h_{i-1} a_i$, $h_i = g_i u^{k_i m}$ 
for $i \geq 1$ (so that $h_l = e$), 
the union of all the paths $[g_j,h_j]$ and $[h_j,g_{j+1}]$ 
is a closed loop in $\mathcal{C}$ and 
(using again the fact that $\mathcal{C}$ is a tree), 
$[g_1,h_1]$ is contained in the union of the other intervals. 
Each $[h_{j-1},g_j]$ has length $\lvert a_j \rvert_S$, 
so for $m$ as in the statement, 
there exists $i \neq 1$ such that $[g_1,h_1]$ overlaps with $[g_i,h_i]$ 
in an interval of length at least $\lvert u^2 \rvert_S$. 

We may write $u$ uniquely as $y^{-1} zy$, 
for $y,z\in\Gamma$ reduced words in $S$, 
with $z$ cyclically reduced and 
$\lvert u \rvert_S = 2 \lvert y \rvert_S + \lvert z \rvert_S$. 
There exist $M , N \in \mathbb{Z}$ with 
$g_1 ^{-1} h_1 = y^{-1} z^M y$ and $g_i ^{-1} h_i = y^{-1} z^N y$ 
. 
Further, 
$I_1 = [g_1 y^{-1},h_1 y^{-1}] = [g_1 y^{-1},g_1 y^{-1} z^M]$ 
overlaps with $I_2 = [g_i y^{-1},h_i y^{-1}] = [g_i y^{-1},g_i y^{-1} z^N]$ 
in an interval of length at least 
$\lvert u^2 \rvert_S-2\lvert y \rvert_S \geq 2\lvert z \rvert_S$. 
Let $p$ and $q$ be the shortest prefixes of $z^M$ and $z^N$, 
respectively, such that $g_1 y^{-1} p = g_i y^{-1} q$ 
(that is, the starting-points of common subinterval of $I_1$ and $I_2$). 

We claim that 
$\lvert p \rvert_S \equiv \lvert q \rvert_S \mod \lvert z \rvert_S$, 
so that the initial common subinverval of $I_1$ and $I_2$ 
of length $\lvert z \rvert_S$ contains a point 
$g_1 y^{-1} z^{k_1 '} = g_i y^{-1} z^{k_2 '}$ 
for some $k_1 ' , k_2 ' \in \mathbb{Z}$, so that: 
\begin{center}
$u^{k_1 ' - k_2 ' - k_1 m} = g_1 ^{-1} g_i = a_2 u^{k_2 m} \cdots u^{k_{i-1} m} a_i$
\end{center}
commutes with $u$, contradicting the inductive hypothesis. 

Suppose the claim fails. 
First suppose $M$ and $N$ have the same sign. 
Comparing the initial common subintervals of $I_1$ and $I_2$ 
of length $2 \lvert z \rvert_S$, 
we see that there exist $v ,w \in \Gamma$ nontrivial reduced words with 
$z= vw = wv$, and $\lvert z \rvert_S = \lvert v \rvert_S + \lvert w \rvert_S$ 
(specifically, $\lvert v \rvert_S \equiv 
\lvert p \rvert_S - \lvert q \rvert_S \mod \lvert z \rvert_S$). 
Since $v$ and $w$ commute, they are powers of a common word, 
contradicting the fact that $u$ (and hence $z$) is not a proper power. 
Similarly if $M$ and $N$ have opposite signs, 
we obtain $z=vw$ and $z^{-1} = wv = wzw^{-1}$, 
but no nontrivial element of a free group is conjugate to its inverse. 
We therefore have the desired claim. 
\end{proof}

\begin{proof}[Proof of Theorem \ref{MIFfreeThm}]
Let $e \neq w \in \Gamma \ast \mathbb{Z}$ with $\lvert w \rvert \leq n$. 
Conjugating, we may assume that $w (x)$ 
has the form given in (\ref{freeprodcycredform}). 
Moreover, 
\begin{equation*}
\sum_{i=1} ^l \lvert a_i \rvert \leq n
\end{equation*}
so that by Proposition \ref{MIFfreePropn} 
$w$ has complexity at most $2 \lvert u \rvert_S n + \lvert u \rvert_S + n$ 
for any $u \in \Gamma$ satisfying the conditions of 
Proposition \ref{MIFfreePropn}. 
It therefore suffices to find such $u$ satisfying $\lvert u \rvert \ll \log n$. 

There exists $C>0$ such that $\lvert B_S (k) \rvert \geq \exp (C k)$ 
for all $k \in \mathbb{N}$. 
Since centralizers in free groups are cyclic, 
the union of the centralizers of the $a_i$ cover at most $(2k+1)l \leq (2k+1)n$ 
points in $B_S (k)$. 
Hence there exists $C' > 0$ such that for all $k \geq C' \log (n)$, 
$B_S(k)$ contains an element $u$ such that $[u,a_i] \neq e$ for all $i$. 
Taking $u$ to be of minimal length among such elements, 
we may assume $u$ is not a proper power. 
\end{proof}

\begin{rmrk}
Unlike the class of lawless groups, the class of MIF groups is not closed 
under taking overgroups. 
For instance $\Gamma \times C_2$ is non-MIF, for any group $\Gamma$. 
There seems to be no straightforward relationship between 
the MIF growth of a group $\Delta$ and that of a finitely generated 
overgroup $\Gamma$ (in the spirit of Lemma \ref{subgrpcomplem}), 
even if $\Gamma$ is itself MIF. 
\end{rmrk}

\section{Open questions} \label{OpenSect}

There are many interesting questions 
one can ask about the spectrum of possible lawlessness 
growth functions. 
For instance one may wonder whether there is a universal 
\emph{upper} bound on $\mathcal{A}_{\Gamma}$ for lawless groups. 

\begin{qu} \label{BigLGQu}
Does there exist, for all non-decreasing functions 
$f : \mathbb{N} \rightarrow \mathbb{N}$, 
a finitely generated lawless group $\Gamma = \Gamma(f)$ 
satisfying $\mathcal{A}_{\Gamma} \npreceq f$? 
Does there exist such a group satisfying 
$\mathcal{A}_{\Gamma} \succeq f$?
\end{qu}

Our main source of strong lower bounds on $\mathcal{A}_{\Gamma}$ 
comes from \emph{upper} bounds on torsion growth for infinite 
finitely generated $p$-groups. 
Under this approach, 
there are viable methods for constructing finitely generated lawless $p$-groups 
which are unrelated to Golod-Shafarevich theory. 
For instance, there are many examples of branch $p$-groups 
beyond Grigorchuk's group, 
and all (weakly) branch groups are lawless 
(by an argument similar to our proof of Theorem \ref{Grigmainthm}). 

\begin{qu} \label{BrachTGQuestion}
How slow can the torsion growth of a finitely generated 
(weakly) branch $p$-group be? Can it be sublinear? 
\end{qu}

For example, by Theorem 7.8 of \cite{BartSuni}, 
Grigorchuk's group has torsion growth at least linear. 
In a completely different direction, 
a result of Dru\c{t}u-Sapir shows that any non-virtually-cyclic 
group which admits an asymptotic cone with a cut-point is lawless. 
By \cite{OlOsSaKaKl}[Theorem 1.12] 
such groups include some finitely generated torsion groups 
(indeed $p$-groups, as is made clear in the proof). 
It is unclear to us at present how effective the construction 
in \cite{OlOsSaKaKl} is able to be made, and whether it could give rise to 
groups of fast lawlessness growth. 

Our proof of the upper bound in Theorem \ref{Grigmainthm}, 
and the possibility of improving upon it, 
naturally connects to the following question 
about finite groups. 

\begin{qu}
Let $W_n$ be the $(n+1)$-fold iterated wreath product of $C_2$, 
as in Section \ref{BranchSect}. 
What is the length of the shortest law in $F_k$ for $W_n$? 
\end{qu}

As far as the author is aware, there is no shorter law known for 
$W_n$ than the obvious power-word $x^{2^{n+1}}$. 

Turning to MIF growth, 
we may ask for examples of groups for which $\mathcal{M}_{\Gamma}$ grows slowly. 
Since MIF is a difficult proerty to satisfy, 
exhibiting such groups may be rather challenging. 
For instance, it is natural (if rather ambitious) 
to ask whether Theorem \ref{MIFLBThm} is sharp. 

\begin{qu}
Does there exist $\Gamma$ with $\mathcal{M}_{\Gamma} (n) \ll \log (n)$? 
\end{qu}

Gromov hyperbolic groups form an important class of MIF groups, 
properly containing the nonabelian finite-rank free groups. 
We may ask for a generalization of the bound 
from Theorem \ref{MIFfreeThm} to this class, 
and for insight into the features of the geometry of groups on which the 
bound depends. 

\begin{prob}
Let $\Gamma$ be a torsion-free Gromov hyperbolic group. 
Give an upper bound on $\mathcal{M}_{\Gamma}$. 
Does there exist, for each $\delta > 0$, 
an increasing function $f_{\delta} : \mathbb{N} \rightarrow \mathbb{N}$ 
such that, if $\Gamma$ admits a $\delta$-hyperbolic Cayley graph, 
then $\mathcal{M}_{\Gamma} \preceq f_{\delta}$? 
\end{prob}

Finally, since words in $F_k$ of a given length are a small subset 
of the elements of $\Gamma \ast F_k$, 
it is clear that for any MIF group $\Gamma$, 
$\mathcal{A}_{\Gamma} \preceq \mathcal{M}_{\Gamma}$, 
and in general one would expect 
it would be a remarkable achievement to identify a group in which 
elements of $\Gamma \ast F_k$ of essentially maximal comlexity 
for their length already occur in $F_k$. 

\begin{qu} \label{LawMIFequivQu}
Does there exist a finitely generated MIF group $\Gamma$ with 
$\mathcal{A}_{\Gamma} \approx \mathcal{M}_{\Gamma}$? 
\end{qu}

By Theorems \ref{bddLGthm} and \ref{MIFLBThm} a group providing a positive answer to 
Question \ref{LawMIFequivQu} would needs must have no $F_2$-subgroups. 
In a forthcoming work, we shall give a finite group analogue of 
a positive answer to Question \ref{LawMIFequivQu}: 
a sequence $(G_n)$ of finite groups having no common law, 
such that for each $n$, 
the length of the shortest law for $G_n$ is comparable to the length of the 
shortest mixed identity. 

Following circulation of a preliminary version of the present article, 
J.M. Petschick has shared with the author a construction 
of a group answering the final part of Question \ref{BrachTGQuestion} 
in the affirmative (as yet unpublished). 
His result therefore also strengthens the conclusion 
of our Theorem \ref{GSmainthm}. 

\subsection*{Acknowledgements}

I am grateful to John Wilson for helpful discussions concerning 
Theorem \ref{bddLGthm}; 
to Henry Wilton for providing a useful reference, 
and to Matt Brin, Francesco Fournier-Facio and Alejandra Garrido for enlightening conversations. 
This work was partially supported by ERC grant no. 648329 ``GRANT''.

\end{document}